\documentclass{amsart}

\usepackage{amssymb,amsmath,latexsym,amsthm}
\usepackage{MnSymbol}
\usepackage{graphicx}
\usepackage{fontenc}%
\usepackage{mathrsfs}
\usepackage{color}
\usepackage{paralist}

\usepackage[verbose]{wrapfig}
\usepackage[english]{babel} 
\usepackage{pstricks-add} 

\newtheorem{theorem}{Theorem}[section]
\newtheorem{definition}[theorem]{Definition}

\newtheorem{example}[theorem]{Example}

\newtheorem{remark}[theorem]{Remark}

\usepackage{pstricks,pst-text,pst-grad,pst-node,pst-3dplot,pstricks-add,pst-poly}
\definecolor{pink}{rgb}{1, .75, .8}
\psset{arrows=->, labelsep=3pt, mnode=circle}

\definecolor{lgrey}{gray}{.85}

\def\defineTColor#1#2{%
 \newpsstyle{#1}{%
  fillstyle=vlines,hatchcolor=#2,
  hatchwidth=0.1\pslinewidth,
  hatchsep=1\pslinewidth}%
  }
\defineTColor{Tgray}{gray}
\defineTColor{Tgreen}{green}
\defineTColor{Tyellow}{yellow}
\defineTColor{Tred}{red} 
\defineTColor{Tblue}{blue} 
\defineTColor{Tmagenta}{magenta}
\defineTColor{Torange}{orange}
\defineTColor{Twhite}{white}
\defineTColor{Tgray}{lgrey}
\defineTColor{Tbgreen}{brightgreen}

\newcommand{\sn}{\mathop{\delta}\limits^{\doublewedge}}
\newcommand{\CL}{\mbox{CL}}
\newcommand{\cl}{\mbox{cl}}
\newcommand{\Int}{\mbox{int}}
\begin{document}

\title[Strongly proximal continuity \& strong connectedness]{Strongly proximal continuity\\ \& strong connectedness}

\author[J.F. Peters]{J.F. Peters$^{\alpha}$}
\email{James.Peters3@umanitoba.ca, cguadagni@unisa.it}
\address{\llap{$^{\alpha}$\,}Computational Intelligence Laboratory,
University of Manitoba, WPG, MB, R3T 5V6, Canada and
Department of Mathematics, Faculty of Arts and Sciences, Ad\i yaman University, 02040 Ad\i yaman, Turkey}
\author[C. Guadagni]{C. Guadagni$^{\beta}$}
\address{\llap{$^{\beta}$\,}Computational Intelligence Laboratory,
University of Manitoba, WPG, MB, R3T 5V6, Canada and
Department of Mathematics, University of Salerno, via Giovanni Paolo II 132, 84084 Fisciano, Salerno , Italy}
\thanks{The research has been supported by the Natural Sciences \&
Engineering Research Council of Canada (NSERC) discovery grant 185986.}


\date{}

\dedicatory{Dedicated to the Memory of Som Naimpally}

\begin{abstract}
This article introduces strongly proximal continuous (s.p.c.) functions, strong proximal equivalence (s.p.e.) and strong connectedness.  A main result is that if topological spaces $X,Y$ are endowed with compatible strong proximities and $f:X\longrightarrow Y$ is a bijective s.p.e., then its extension on the hyperspaces $\CL(X)$ and $\CL(Y)$, endowed with the related strongly hit and miss hypertopologies, is a homeomorphism.  For a topological space endowed with a strongly near proximity, strongly proximal connectedness implies connectedness but not conversely.  Conditions required for strongly proximal connectedness are given.  Applications of s.p.c. and strongly proximal connectedness are given in terms of strongly proximal descriptive proximity.
\end{abstract}


\maketitle

\section{Introduction}
This article carries forward recent work on strong proximity~\cite{PetersGuadagni2015stronglyNear,PetersGuadagni2015stronglyFar,Peters2015visibility,Peters2015PJMS} and strongly hit and miss hypertopologies~\cite{PetersGuadagni2015hypertopologies}.  Strongly proximal continuous functions and strongly near connectedness of subsets in topological spaces are introduced.

\section{Preliminaries}
Proximities are a powerful tool to deal with the concept of nearness without involving metrics (see, {\em e.g.},\cite{DiConcilio2009, Naimpally2009,DiConcilio2000SetOpen,DiConcilio2000PartialMaps}). Proximities are binary relations on the power set $\mathscr{P}(X)$ of a nonempty set $X$.  $A\ \delta\ B$ reads \emph{$A$ is near $B$}. From the usual proximity space axioms,  it suffices to have $A \cap B \neq \emptyset$ to obtain $A\ \delta\ B$.  We require something more. We want to talk about a stronger kind of nearness. For this reason we introduced \emph{strong proximities} in \cite{PetersGuadagni2015stronglyNear}.  Strong proximities satisfy the following axioms.

\begin{definition}
Let $X$ be a topological space, $A, B, C \subset X$ and $x \in X$.  The relation $\sn$ on $\mathscr{P}(X)$ is a \emph{strong proximity}, provided it satisfies the following axioms.
\vspace{3mm}
\begin{description}
\item[(N0)] $\emptyset \not\sn A, \forall A \subset X $, and \ $X \sn A, \forall A \subset X$
\item[(N1)] $A \sn B \Leftrightarrow B \sn A$
\item[(N2)] $A \sn B \Rightarrow A \cap B \neq \emptyset$
\item[(N3)] If $\{B_i\}_{i \in I}$ is an arbitrary family of subsets of $X$ and  $A \sn B_{i^*}$ for some $i^* \in I \ $ such that $\Int(B_{i^*})\neq \emptyset$, then $  \ A \sn (\bigcup_{i \in I} B_i)$
\item[(N4)] $\mbox{int}A \cap \mbox{int} B \neq \emptyset \Rightarrow A \sn B$ \qquad \textcolor{blue}{$\blacksquare$}
\end{description}
\end{definition}

\noindent When we write $A \sn B$, we read $A$ is \emph{strongly near} $B$.
For each \emph{strong proximity}, we assume the following relations:
\begin{description}
\item[N5)] $x \in \Int (A) \Rightarrow x \sn A$ 
\item[N6)] $\{x\} \sn \{y\} \Leftrightarrow x=y$  \qquad \textcolor{blue}{$\blacksquare$} 
\end{description}
So, for example, if we take the strong proximity related to non-empty intersection of interiors, we have that $A \sn B \Leftrightarrow \Int A \cap \Int B \neq \emptyset$ or either $A$ or $B$ is equal to $X$, provided $A$ and $B$ are not singletons; if $A = \{x\}$, then $x \in \Int(B)$, and if $B$ too is a singleton, then $x=y$. It turns out that if $A \subset X$ is an open set, then each point that belongs to $A$ is strongly near $A$.

Related to this new kind of nearness introduced in~\cite{PetersGuadagni2015stronglyNear} which extends traditional proximity (see, {\em e.g.},~\cite{Naimpally1970,Lodato1962,Lodato1964,Lodato1966,Naimpally2013,Peters2012notices}), we defined a new kind of \emph{hit-and-miss hypertopology}~\cite{PetersGuadagni2015stronglyNear, PetersGuadagni2015hypertopologies}, which extends recent work on hypertopologies (see, {\em e.g.},~\cite{Beer1993hit,DiConcilio1989,DiConcilio2013action,DiMaio2008hypertop,DiMaio1995hypertop,DiMaio1992hypertop,Som2006hypertopology,Guadagni2015,Naimpally2002}).  The important thing to notice that this work has its foundation in geometry~\cite{Guadagni2015,Peters2015VoronoiAMSJ,Peters2015visibility}.

The \emph{strongly hit and far-miss} topology $\tau^\doublewedge$ has as subbase the sets of the form:
\begin{itemize}
\item $V^{\doublewedge} = \{E \in \CL(X): E \sn V \}$, where $V$ is an open subset of $X$,
\item $A^{++} =  \{ \ E \in \CL(X) : E \not\delta X\setminus  A  \ \}$, where $A$ is an open subset of $X$.
\end{itemize}

\noindent In \cite{PetersGuadagni2015hypertopologies}, we considered the Hausdorffness of the previous topology associated with suitable families of subsets.

In this paper we go deeper into the study of \emph{strong proximities} in terms of concepts of strong proximal continuity and strong proximal connectedness. Moreover, these new forms of proximal continuity and connectedness are applied in some examples of descriptive nearness, which is particularly useful for many applications.

\section{Strongly proximal continuity}

After introducing strong proximities, the natural continuation is to look at mappings that preserve proximal structures. We call such \emph{strongly proximal continuous} mappings.

\begin{definition}
Suppose that $(X, \tau_X, {\sn}_X) $ and $(Y, \tau_Y, {\sn}_Y)$ are topological spaces endowed with strong proximities. 
\begin{description}
\item[Strongly proximal continuous (s.p.c.)] $\mbox{}$\\
We say that $f: X \rightarrow Y$ is \emph{strongly proximal continuous } and we write \emph{s.p.c.} if and only if, for $A, B \subset X$, 
\[
\ A\ {\sn}_X\ B \Rightarrow f(A)\ {\sn}_Y\ f(B).
\] 
\item[Strongly proximal equivalence (s.p.e.)] $\mbox{}$\\ 
If, for $A, B \subset X$, 
\[ 
A\ {\sn}_X\ B \Leftrightarrow f(A)\ {\sn}_Y\ f(B),
\]
then $f$ is a \emph{strongly proximal equivalence } and we write \emph{s.p.e.} \qquad \textcolor{blue}{$\blacksquare$}
\end{description}
\end{definition}

By the definition of strong proximities we know that, if we have ${\sn}_X$ on $(X, \tau_X)$,  for each open set $A \subset X$, it holds that each point that belongs to $A$ is strongly near $A$ in ${\sn}_X$. So it seems reasonable to ask if the strong proximity can generate a topology by defining as open sets those for which we have that each point that belongs to them is strongly near to them. This is possible if we require  that :
\[ x \sn A, x \sn B \ \Rightarrow x \sn (A \cap B)\]

In addition, this topology could be different from the starting topology. When they match, we say that the strong proximity is \emph{compatible} with the starting topology.
Next, we want to investigate on some relations involving strongly proximal continuity.

\begin{theorem}\label{open}
Suppose that $(X, \tau_X, {\sn}_X) $ and $(Y, \tau_Y, {\sn}_Y)$ are topological spaces endowed with compatible strong proximities and $f: X \rightarrow Y$ is s.p.c. . Then $f$ is an open mapping, that is $f$ maps open sets in open sets.
\end{theorem}
\begin{proof}
Consider $V$ open subset of $X$. It means that each point that belongs to $V$ is strongly near $V$ in ${\sn}_X$. Now suppose that $f(x)$ is in $f(V)$. We want to show that $f(x) {\sn}_Y f(V)$. But it is true by the strongly proximal continuity of $f$.
\end{proof}

\setlength{\intextsep}{0pt}
\begin{wrapfigure}[11]{R}{0.35\textwidth}
\begin{minipage}{3.5 cm}
\begin{center}
\begin{pspicture}
(0.0,3.5)(2.5,4.0)
\psset{yunit=0.5,xunit=0.5}
\psaxes{->}(0,0)(-0.0,-0.0)(5.5,4.5)
\psset{algebraic,plotpoints=501}
\pstriangle[linecolor=green,linestyle=solid,linewidth=0.05,style=Tgreen](1.5,0.00)(3.00,2.4)
\psrotate(3.0,0.00){80}{%
\pstriangle[linecolor=orange,linestyle=solid,linewidth=0.05,style=Torange](4.50,0.00)(3.00,2.4)}
\psrotate(1.5,1.80){160}{%
\pstriangle[linecolor=blue,linestyle=solid,linewidth=0.05,style=Tblue](1.5,0.00)(3.00,2.4)}
\rput(1.50,0.55){\footnotesize $\boldsymbol{A{_1}}$}
\rput(2.60,1.55){\footnotesize $\boldsymbol{A_2}$}
\rput(1.8,2.85){\footnotesize $\boldsymbol{A_3}$}
\rput(0.28,-1.55){\qquad\qquad\qquad\qquad\footnotesize 
                 $\boldsymbol{\mbox{Fig.}\ 1.1\  80^o Rotations}$}
								\label{fig:rotate80}
 \end{pspicture}
\end{center}
\end{minipage}
\end{wrapfigure}

\begin{remark}\label{remark1}
Observe that the converse is not in general true. For example take $(X, \tau_X, {\sn}_X) = (\mathbb{R}^2, \tau_e, {\sn}_X)$ and $(Y, \tau_Y, {\sn}_Y) = (\mathbb{R}^2, \tau_e, {\sn}_Y)$, where $\tau_e$ is the Euclidean topology, $A {\sn}_X B \Leftrightarrow A \cap\ \Int (B) \neq \emptyset$ or $\Int (A) \cap B \neq \emptyset$ or either $A$ or $B$ is equal to $X$, provided that $A$ and $B$ are not both singletons, and if $A= \{x\}$ and $B= \{y\}$, $x = y$. Finally $A {\sn}_Y B \Leftrightarrow \Int A \cap \Int B \neq \emptyset$, 
provided $A$ and $B$ are not singletons; if $A = \{x\}$, then $x \in \Int(B)$, and if $B$ too is a singleton, then $x=y$. 
In this case, if we take the identity map, it is open but not s.p.c.  \qquad \textcolor{blue}{$\blacksquare$}
\end{remark}

In \cite{PetersGuadagni2015stronglyNear}, we introduced a particular hit and miss hypertopology, a strongly hit and far miss hypertopology. If $\sn$ is a strong proximity we can look at the strongly hit and miss hypertopology on $\CL(X)$, $\tau^\doublewedge$, having as subbase the sets of the form:

\begin{itemize}
\item $V^{\doublewedge} = \{E \in \CL(X): E \sn V \}$, where $V$ is an open subset of $X$,
\item $A^{+} =  \{ \ E \in \CL(X) : E \cap (X\setminus  A) = \emptyset \ \}$, where $A$ is an open subset of $X$.
\end{itemize}

If we have a function $f: (X, \tau_X, {\sn}_X) \rightarrow (Y, \tau_Y, {\sn}_Y)$ that preserves closed subsets, we can also consider a function on $(\CL(X), \tau^\doublewedge_X)$ into $(\CL(Y), \tau^\doublewedge_Y)$. We indicate this as $f^\doublewedge$. Next we want to highlight a particular relation existing between $f$ and $f^\doublewedge$.

\begin{figure}[!ht]
\begin{center}
\begin{pspicture}
 (3.0,-0.5)(2.5,1.5)
\psset{yunit=0.5,xunit=0.5}
\psaxes{->}(0,0)(-0.0,-0.0)(12.5,2.5)
\psset{algebraic,plotpoints=501}
\pstriangle[linecolor=green,linestyle=solid,linewidth=0.05,style=Tgreen](1.5,0.00)(3.00,2.4)
\psarc{->}(3.0,1.0){0.3}{10}{170}
\pstriangle[linecolor=orange,linestyle=solid,linewidth=0.05,style=Torange](4.50,0.00)(3.00,2.4)
\psarc{->}(6.0,1.0){0.3}{10}{170}
\pstriangle[linecolor=blue,linestyle=solid,linewidth=0.05,style=Tblue](7.50,0.00)(3.00,2.4)
\pstriangle[linecolor=gray,linestyle=dotted,linewidth=0.025,style=Tgray](10.50,0.00)(3.00,2.4)
\rput(1.50,0.85){\footnotesize $\boldsymbol{A{_1}}$}
\rput(4.50,0.85){\footnotesize $\boldsymbol{A_2}$}
\rput(7.5,0.85){\footnotesize \textcolor{black}{$\boldsymbol{A_3}$}}
 \end{pspicture}
\caption{Family of Adjacent Triangles}
\label{fig:TriangleRotations}
\end{center}
\end{figure}
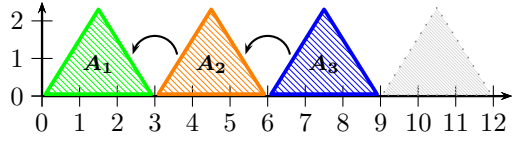

Even in the case in which the strong proximity does not generate a compatible topology, it could be interesting to look at s.p.c. functions. In particular, we focus our attention on a particular family of subsets.

\begin{definition}
Suppose that $(X, \tau_X, {\sn}_X) $ and $(Y, \tau_Y, {\sn}_Y)$ are topological spaces endowed with strong proximities and $\mathscr{S}$ is a family of subsets of $X$. We say that $f: X \rightarrow Y$ is \emph{strongly proximal continuous on $\mathscr{S}$} if and only if  $\forall A, B \in \mathscr{S}, \ A {\sn}_X B \Rightarrow f(A) {\sn}_Y f(B)$.  \qquad \textcolor{blue}{$\blacksquare$}
\end{definition}

\begin{example}\label{ex1}
Take $(X, \tau_X, {\sn}_X) = (\mathbb{R}^2, \tau_e, {\sn}_X)$ and $(Y, \tau_Y, {\sn}_Y) = (\mathbb{R}^2, \tau_e, {\sn}_Y)$, where $\tau_e$ is the Euclidean topology, $A {\sn}_X B \Leftrightarrow A \cap B \neq \emptyset$, and $A {\sn}_Y B \Leftrightarrow \Int A \cap \Int B \neq \emptyset$, 
provided $A$ and $B$ are not singletons; if $A = \{x\}$, then $x \in \Int(B)$, and if $B$ too is a singleton, then $x=y$. Consider the triangles as in Fig.~\ref{fig:TriangleRotations} and let $\mathscr{S}$ be the family of those triangles. Define a function $g: X \rightarrow Y$ as follows:
\[
g(x)=
\begin{cases}
x, & \text{if $x \leq x_1$,} \\
R(x_1, 80^\circ), & \text{if $x_1 \leq x \leq x_2$,}\\
R(x_2,80^\circ) \circ R(x_1, 80^\circ), & \text{if $x_2 \leq x \leq x_3$,}\\
R(x_3,80^\circ) \circ R(x_2,80^\circ) \circ R(x_1, 80^\circ), & \text{if $x_3 \leq x \leq x_4$,}\\

\prod_{i \in \{n, n-1..,1\}} R(x_i,80^\circ),& \text{if $x_n \leq x \leq x_{n+1}$.}
\end{cases}
\]
where $R(x_i, 80^\circ)$ is the rotation around the point $x_i$ by $80^\circ$, and by \\
 $\prod_{i \in \{n, n-1..,1\}} R(x_i,80^\circ)$ we mean the composition of rotations (see Fig. 1.1).
We have that $g$ is s.p.c. on $\mathscr{S}$.  \qquad \textcolor{blue}{$\blacksquare$}
\end{example}

\begin{theorem}
Let $(X, \tau_X, {\sn}_X) $ and $(Y, \tau_Y, {\sn}_Y)$ be topological spaces endowed with compatible strong proximities. If $f$ is a bijective s.p.e., then $f^\doublewedge$ is an homeomorphism.
\end{theorem}
\begin{proof}
First of all observe that, being $f$ a bijective s.p.e., it is closed by Theorem~\ref{open}. So we can think at that as a function between the hyperspaces. Now we want to prove that $f^\doublewedge$ is open. Take an open subset of $\tau^\doublewedge_X$, \ $\mathcal{U} = H^+ \cap (A^\doublewedge_1 \cap...\cap A^\doublewedge_n)$, where $H, A_1,..., A_n$ are open subsets of $X$. Suppose that $E$ in $\CL(X)$ belongs to $\mathcal{U}$. This means that $E \ {\sn}_X \  A_i, i \in \{1,..,n \}$ and $E \cap (X \setminus H)= \emptyset$. By the hypothesis and by Theorem~\ref{open} we have that $f(E) \in \mathcal{V}= f(H)^+ \cap (f(A_1)^\doublewedge \cap...\cap f(A_n)^\doublewedge)$, where $f(H), f(A_1),...,f(A_n)$ are open subsets of $Y$. So $f(\mathcal{U}) \subset \mathcal{V}$. Conversely, if $D$ is a closed subset of $Y$ that belongs to $\mathcal{V}$, it means that $D \ {\sn}_Y \ f(A_i), \ i \in \{1,..,n \}$ and $D \cap (Y \setminus f(H)) = \emptyset $. In the same way as before we obtain that $f^{-1}(D) \in \mathcal{U}$. Hence, $f(\mathcal{U})= \mathcal{V}$ and $f^{\doublewedge}$ is open. Moreover, it is easily seen that by the hypothesis we obtain also that $f^{\doublewedge}$ is bijective. Finally, by applying the same procedure to $({f^\doublewedge})^{-1}$, we have that $({f^\doublewedge})^{-1}$ is open, too, and so ${f^\doublewedge}$ is an homeomorphism.
\end{proof}

\begin{figure}[!ht]
\begin{center}
\begin{pspicture}
 (4.5,-1.5)(2.5,2.5)
\psset{yunit=1.0,xunit=1.0}
\psaxes{->}(0,0)(-2.5,-2.5)(10.5,2.5)
\psset{algebraic,plotpoints=501}
\pscircle[linestyle=solid,linestyle=solid,linewidth=0.05,style=Twhite](0.0,0.00){1.50}
\pscircle[linecolor=black,linestyle=solid,linewidth=0.05,fillstyle=solid,fillcolor=green,opacity=0.5]
  (0.0,0.00){1.00}
\pscircle[linestyle=solid,linestyle=dotted,linewidth=0.05,style=Tyellow](2.5,0.00){1.50}
\pscircle[linestyle=solid,linecolor=white,linestyle=dotted,linewidth=0.05,fillstyle=solid,fillcolor=black](7.5,0.00){1.50}
\pscircle[linestyle=solid,linestyle=solid,linewidth=0.05,style=Twhite](5.0,0.00){1.50}
\pscircle[linestyle=solid,linestyle=solid,linewidth=0.05,style=Tblue](5.0,0.00){1.00}
\rput(-1.00,0.65){\footnotesize $\boldsymbol{A{_1}}$}
\rput(-0.35,0.55){\footnotesize $\boldsymbol{A_1^{\prime}}$}
\rput(2.5,0.55){\footnotesize $\boldsymbol{A_2}$}
\rput(5.0,1.25){\footnotesize $\boldsymbol{A_3}$}
\rput(5.0,0.55){\footnotesize $\boldsymbol{A_3^{\prime}}$} 
\rput(7.5,0.55){\footnotesize \textcolor{white}{$\boldsymbol{A_4}$}}
 \end{pspicture}
\caption{The Family $\mathscr{B}$ Containing Open and Closed Disks}
\label{fig:figinv1}
\end{center}
\end{figure}

\begin{figure}[!ht]
\begin{center}
\begin{pspicture}
 (-3.5,-2.5)(2.5,2.5)
\psset{algebraic,plotpoints=501}
\pscircle[linecolor=black,linestyle=dotted,linewidth=0.025,fillstyle=solid,fillcolor=green,opacity=0.9]
 (-0.0,0.00){2.30}
\pscircle[linestyle=solid,linestyle=dotted,linewidth=0.025,style=Tyellow](1.5,0.00){0.80}
\pscircle[linestyle=solid,linestyle=solid,linewidth=0.025](0.635,0.00){0.30}
\pscircle[linestyle=solid,linestyle=dotted,linewidth=0.025,style=Tblue](0.535,0.00){0.18}
\pscircle[linestyle=solid,linecolor=white,linestyle=dotted,linewidth=0.025,fillstyle=solid,fillcolor=black]
(0.29,0.00){0.10}
\rput(-1.50,0.85){\tiny $\boldsymbol{i_1(A{_1})}$}
\rput(1.5,0.50){\tiny $\boldsymbol{i_1(A{_2})}$}
\psline[linecolor=black,linestyle=solid,linewidth=0.025]{->}(0.335,0.80)(0.50,0.25)
\rput(0.335,0.90){\tiny $\boldsymbol{i_1(A{_4})}$}
\rput(0.05,-0.90){\tiny $\boldsymbol{i_1(A{_3})}$}
\psline[linecolor=black,linestyle=solid,linewidth=0.025]{->}(0.05,-0.80)(0.30,-0.10)
 \end{pspicture}
\caption{Application of $i_1$ on the Family $\mathscr{B}$}
\label{fig:figinv2}
\end{center}
\end{figure}

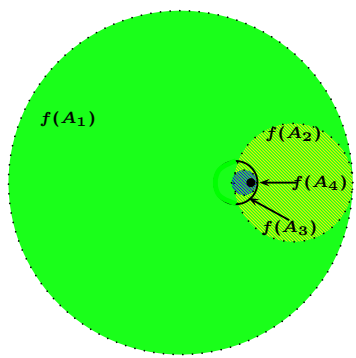
\begin{figure}[!ht]
\begin{center}
\begin{pspicture}
 (-3.5,-2.5)(2.5,2.5)
\psset{algebraic,plotpoints=501}
\pscircle[linecolor=black,linestyle=dotted,linewidth=0.025,fillstyle=solid,fillcolor=green,opacity=0.9]
 (-0.0,0.00){2.30}
\pscircle[linecolor=black,linestyle=dotted,linewidth=0.025,style=Tyellow](1.5,0.00){0.80}
\pscircle[linestyle=solid,linewidth=0.025,linecolor=black](0.735,0.00){0.30}
\pscurve[linestyle=solid,linewidth=0.065,linecolor=green,opacity=0.9]{-}%
 (0.735,-0.28)(0.5,-0.15)(0.5,0.15)(0.735,0.29)
\pscircle[linecolor=black,linestyle=dotted,linewidth=0.025,style=Tblue](0.845,0.00){0.18}
\pscircle[linestyle=solid,linestyle=solid,linewidth=0.055,fillcolor=black](0.935,0.00){0.06}
\rput(-1.50,0.85){\tiny $\boldsymbol{f(A{_1})}$}
\rput(1.5,0.65){\tiny $\boldsymbol{f(A{_2})}$}
\psline[linecolor=black,linestyle=solid,linewidth=0.025]{->}(1.445,-0.50)(0.90,-0.20)
\rput(1.445,-0.60){\tiny $\boldsymbol{f(A{_3})}$}
\psline[linecolor=black,linestyle=solid,linewidth=0.025]{->}(1.545,0.00)(1.045,0.00)
\rput(1.845,0.00){\tiny $\boldsymbol{f(A{_4})}$}
 \end{pspicture}
\caption{Application of s.p.c. $f$ on the Family $\mathscr{B}$}
\label{fig:figinv3}
\end{center}
\end{figure}

\begin{example}
Consider $(X, \tau_X, {\sn}_X)$ and $(Y, \tau_Y, {\sn}_Y) $ as in the remark \ref{remark1}. Let $\mathscr{B}$ be the family of subsets in the Fig.~\ref{fig:figinv1}, that is we have some open disks ($A_2$, $A_4$) and unions ($A_1$, $A_3$) of closed disks ($A_1^{\prime}$, $A_3^{\prime}$) with concentric circumferences. In this case we want to take as transformation the so called circle inversion (see~\cite{Wolfram2015inversion}). 
It is a special reflection with respect to a circle, the inversion circle. It is particularly relevant dealing with hyperbolic geometry. In fact, if we take the \emph{Poincar\' e disk}, some of these transformations reveal that they are among the hyperbolic isometries mapping the disk into itself. 

The general equations for the inverse of the point $(x,y)$ relative to the inversion circle with inversion center $(x_0,y_0)$ and inversion radius $k$ are given by
\begin{align*}
x' &= x_0+ \dfrac{ k^2(x-x_0)}{(x-x_0)^2+(y-y_0)^2},\\
y' &= y_0+ \dfrac{ k^2(y-y_0)}{(x-x_0)^2+(y-y_0)^2}.
\end{align*}
By this transformation any point that is inside the inversion circle is mapped outside and vice versa.

Our function $f$ is obtained by composition of inversions: $f= i_3 \circ i_2 \circ i_1$. The function $i_1$ works as the inversion relative to the circle $A^{\prime}_1$ with center $(0,0)$ and radius $2$ on the region outside $A^{\prime}_1$, and as the identity elsewhere (see Fig.~\ref{fig:figinv2}); the function $i_2$ is the inversion with respect to $i_1(A_2)$ on the region inside $i_1(A_1)$ and outside $i_1(A_2)$, and it is the identity elsewhere; $i_3$ is the inversion with respect to $i_2(A^{\prime}_3)$ on the region inside $i_2(A_2)$ and outside $i_2(A_3)$, and it is the identity elsewhere.  We obtain that $f$ is s.p.c. on $\mathscr{B}$.  See Fig.~\ref{fig:figinv3}.  \qquad \textcolor{blue}{$\blacksquare$}
\end{example}

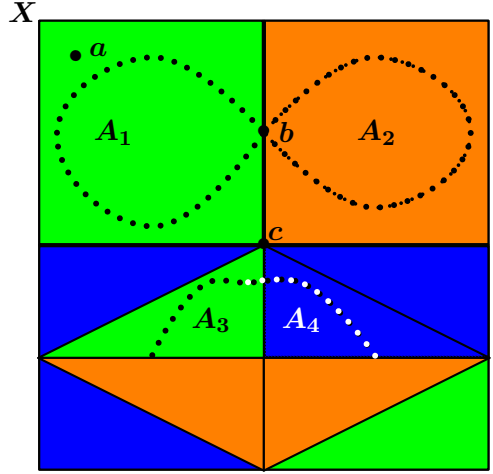
\begin{figure}[!ht]
\begin{center}
\begin{pspicture}
 (-3.0,1.3)(8.0,7.5)
\psframe[linecolor=black](-0.5,1.0)(5.5,7.0)
\psframe[linecolor=black,fillstyle=solid,fillcolor=green](-0.5,4.0)(2.5,7.0)
\psframe[linecolor=black,fillstyle=solid,fillcolor=blue](-0.5,4.0)(2.5,1.0)
\psframe[linecolor=black,fillstyle=solid,fillcolor=orange](2.5,4.0)(5.5,7.0)
\psframe[linecolor=black,fillstyle=solid,fillcolor=green](2.5,4.0)(5.5,1.0) 
\psframe[linecolor=black,fillstyle=solid,fillcolor=blue](2.5,2.5)(5.5,4.0)
\pspolygon[linecolor=black,style=Tblue](2.5,2.5)(2.5,4.0)(5.5,2.5)
\qline(2.5,7.0)(2.5,1.0)\qline(-0.5,4.0)(5.5,4.0)
\qline(-0.5,2.5)(5.5,2.5)
\pspolygon[linecolor=black,fillstyle=solid,fillcolor=green](-0.5,2.5)(2.5,4.0)(2.5,2.5)
\pspolygon[linecolor=black,fillstyle=solid,fillcolor=orange](-0.5,2.5)(2.5,2.5)(2.5,1.0)
\pspolygon[linecolor=black,fillstyle=solid,fillcolor=orange](5.5,2.5)(2.5,2.5)(2.5,1.0)
\pscurve[linestyle=dotted,linewidth=0.05,showpoints=false]{-}%
(2.5,5.5)(4.0,6.5)(5.25,5.5)(4.0,4.5)(2.5,5.5)
\pscurve[linestyle=dotted,linewidth=0.08,showpoints=false]{-}%
(2.5,5.5)(4.0,6.5)(5.25,5.5)(4.0,4.5)(2.5,5.5)
\pscurve[linestyle=dotted,linewidth=0.08,showpoints=false]{-}%
(2.5,5.5)(1.0,6.5)(-0.25,5.5)(1.0,4.25)(2.5,5.5)
\pscurve[linestyle=dotted,linewidth=0.08,showpoints=false]{-}%
(1.0,2.5)(1.75,3.5)(2.25,3.5)(3.0,3.50)(4.0,2.5)
\pscurve[linestyle=dotted,linecolor=white,linewidth=0.08,showpoints=false]{-}%
(2.25,3.5)(3.0,3.50)(4.0,2.5)
\rput(2.5,4.0){\large  $\boldsymbol{\bullet}$}
\rput(2.65,4.15){\large $\boldsymbol{c}$}
\rput(2.5,5.5){\large  $\boldsymbol{\bullet}$}
\rput(2.8,5.5){\large  $\boldsymbol{b}$}
\rput(0.0,6.5){\large  $\boldsymbol{\bullet}$}
\rput(0.3,6.6){\large $\boldsymbol{a}$}
\rput(-0.70,7.10){\large  $\boldsymbol{X}$}
\rput(0.50,5.50){\large  $\boldsymbol{A_1}$}
\rput(4.00,5.50){\large  $\boldsymbol{A_2}$}
\rput(1.80,3.00){\large  $\boldsymbol{A_3}$}
\rput(3.00,3.00){\large  \textcolor{white}{$\boldsymbol{A_4}$}}
 \end{pspicture}
\caption{Spatial Kind of Descriptive Proximity}
\label{fig:spatialDescriptiveProximity}
\end{center}
\end{figure}

\begin{example} {\bf Spatial Kind of Descriptive Nearness}.\\
In this example, we use a spatial kind of \emph{descriptive nearness}. The theory of descriptive nearness~\cite{Peters2014TopologyBook,Peters2013MSClocalNearSets} is usually adopted when dealing with subsets that share some common properties even though the subsets are not really close. We talk about \emph{non-abstract points} when points have locations and features that can be measured. This theory is particularly relevant when we want to focus on some of these aspects of points and sets of points that are known both spatially and descriptively. 

Consider as space $X$ the tessellation of a subspace of $\mathbb{R}^2$ endowed with the Euclidean topology as shown in Fig.~\ref{fig:spatialDescriptiveProximity}. Let $Y = \{g,\ r, \ b \}$ be the set of colors green, red and blue and take $\mathscr{P}(Y)$ endowed with the topology $\tau$ having as base $\mathscr{F}$ defined by
\begin{align*}
\mathscr{F} &= \{ \{g,r\}, \ \{b,r\}, \ \{g,b\}, \ \{r,g,b\}, \\
            &= \{\{r,g,b\}, \{r,g\}, \{r\} \}, \ \{\{r,g,b\}, \{r,g\}, \{g\} \},\\ 
						&= \ \{\{r,g,b\}, \{r,g\}, \{b\} \} \}.\ \mbox{($\tau$-base)}
\end{align*}
Then consider as strong proximities the same of Example~\ref{ex1}. Moreover suppose that in each instant we can have only one subset included in each region. Define a function $f$ in the following way. If $x$ is an internal point of some region, $f(x)$ is the singleton containing the color of the region; if $x$ is on the boundary of some regions, $f(x)$ is the set containing the color of these regions. For example, with respect to the figure, we have $f(a)= \{g\}$, $f(b)= \{g,r\}$ and $f(c)= \{g,r,b \}$. Now, if we take two subsets $A_i(t*), A_j(t*)$ existing in the same instant $t*$, we have that $A_i(t*) \cap A_j(t*)\neq \emptyset \Rightarrow \Int(f(A_i(t*)))\cap \Int(f( A_j(t*))) \neq \emptyset$. So $f$ is s.p.c. on the family of subsets existing in the same instant. \qquad \textcolor{blue}{$\blacksquare$}
\end{example}

\section{$\sn-$connectedness}

Looking at connectedness and its properties, it appears quite natural to try to generalize this concept using strong proximities. Actually, we obtain a strengthening of the standard concept.

Recall the following property, \cite{Willard1970}.

\begin{theorem}\label{connWillard}
If $X = \bigcup_{n=1}^{\infty} X_n$ where each $X_n$ is connected and $X_{n-1} \cap X_n \neq \emptyset $  for each $n \geq 2$, then $X$ is connected.
\end{theorem}

\begin{figure}[!ht]
\begin{center}
\begin{pspicture}
 (-3.5,-0.8)(2.5,2.5) 
\psset{algebraic,plotpoints=501}
%
\pscircle[linecolor=black,linestyle=dotted,linewidth=0.025,style=Tgreen](-2.46,0.50){0.8}
\pscircle[linecolor=black,linestyle=dotted,linewidth=0.025,style=Tgreen](-0.8,0.50){1.3}
\rput(-3.5,1.50){\large $\boldsymbol{X = X_1\cup X_2}$}
\rput(-0.5,0.50){\large $\boldsymbol{X_2}$}
\rput(-2.46,0.50){\large $\boldsymbol{X_1}$}
\end{pspicture}
\caption{\bf $\boldsymbol{\sn}$-Connected Topological Space}
\label{fig:stronglyConnected}
\end{center}
\end{figure}
We define the following new kind of connectedness.

\begin{definition} {\bf $\boldsymbol{\sn}$-Connected Topological Space}.\\
Let $X$ be a topological space and $\sn$ a strong proximity on $X$. We say that $X$ is $\sn-$connected if and only if $X = \bigcup_{i \in I} X_i$, where $I$ is a countable subset of $\mathbb{N}$, $X_i$ and $\Int(X_i)$ are connected for each $i \in I$, and $X_{i-1} \sn X_i$ for each $i \geq 2$.  \qquad \textcolor{blue}{$\blacksquare$}
\end{definition}

\begin{example}
A $\sn$ (strongly connected) topological space is represented in Fig.~\ref{fig:stronglyConnected}.
\qquad \textcolor{blue}{$\blacksquare$}
\end{example}

\begin{theorem}\label{relation}
Let $X$ be a topological space and $\sn$ a strong proximity on $X$. Then $\sn-$connectedness implies connectedness.
\end{theorem}
\begin{proof}
This simply follows by Theorem~\ref{connWillard} and axiom $(N2)$ in the definition of strong proximities.
\end{proof}

The following example shows that the converse is not always true.

\begin{figure}[!ht]
\begin{center}
\begin{pspicture}
 (-3.5,-0.8)(2.5,2.5) 
\psset{algebraic,plotpoints=501}
%
\psframe(-3.8,-1.00)(1.5,2.00)
\pscircle[linecolor=black,linestyle=solid,linewidth=0.025,style=Torange](-2.46,0.50){0.8}
\pscircle[linecolor=black,linestyle=solid,linewidth=0.025,style=Torange](-0.5,0.50){1.2}
\rput(-1.69,0.50){\tiny $\boldsymbol{\bullet}$}
\rput(-2.46,0.50){\large $\boldsymbol{E}$}
\rput(-4.0,2.00){\large $\boldsymbol{\mathbb{R}^2}$}
\end{pspicture}
\caption{Connected Subset {\large $\boldsymbol{E}$ but not $\sn$-connected}}
\label{fig:strongProximity}
\end{center}
\end{figure}

\begin{example}
Consider the subset $E$ of $\mathbb{R}^2$ as shown in Fig.~\ref{fig:strongProximity} and the strong proximity given by $A {\sn} B \Leftrightarrow \Int A \cap \Int B \neq \emptyset$ or either $A$ or $B$ is equal to $X$, provided $A$ and $B$ are not singletons; if $A = \{x\}$, then $x \in \Int(B)$, and if $B$ is also a singleton, then $x=y$. The subset $E$ is connected but it is not $\sn-$connected.  \qquad \textcolor{blue}{$\blacksquare$}
\end{example}

\begin{figure}[!ht]
\begin{center}
\begin{pspicture}
 (-3.5,-0.8)(2.5,2.5) 
\psset{algebraic,plotpoints=501}
%
\psframe(-3.8,-1.00)(3.0,2.00)
\pscircle[linestyle=dotted,linewidth=0.025,fillstyle=solid,fillcolor=orange](-2.46,1.00){0.8}
\pscircle[fillstyle=solid,fillcolor=green](1.5,0.50){1.2}
\rput(1.5,0.50){\large $\boldsymbol{B}$}
\rput(-2.46,1.00){\large $\boldsymbol{A}$}
\end{pspicture}
\caption{Disjoint Balls, both $\sn-$connected}
\label{fig:disjointBalls}
\end{center}
\end{figure}
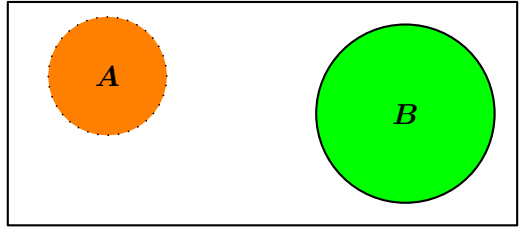

\begin{figure}[!ht]
\begin{center}
\begin{pspicture}
 (-3.5,-0.8)(2.5,2.5) 
\psset{algebraic,plotpoints=501}
%
\psframe[fillstyle=penrose*,psscale=0.3,fillcolor=green!30,hatchcolor=black,opacity=0.3](-3.5,0.50)(-1,2)
\psframe[fillstyle=penrose*,psscale=0.3,fillcolor=green!30,hatchcolor=black,opacity=0.3](-1.8,0.50)(1,2)
\rput(-1.30,1.25){\Large \textcolor{red}{$\boldsymbol{C}$}}
\psframe[fillstyle=solid,fillcolor=white,opacity=0.8](-3.0,0.50)(-2.5,1)
\psframe[fillstyle=solid,fillcolor=white,opacity=0.8](-0.3,0.50)(0.3,1.3)
\psframe[fillstyle=penrose*,psscale=0.3,fillcolor=orange!50,hatchcolor=black,opacity=0.3](-3.0,-1.00)(0.3,0.51)
\rput(-1.30,-0.55){\Large \textcolor{red}{$\boldsymbol{D}$}}
\psframe[fillstyle=solid,fillcolor=white](-2.55,-0.20)(-0.28,0.50)
\qline(-2.8,0.8)(1.5,-0.25)
\qline(-0.2,0.8)(1.5,-0.25)
\rput(2.0,-0.25){\normalsize $\boldsymbol{C\cap D}$}
%
\end{pspicture}
\caption{$C,D$ $\sn$-connected but not $C\cap D$}
\label{fig:strongProximalConnected}
\end{center}
\end{figure}
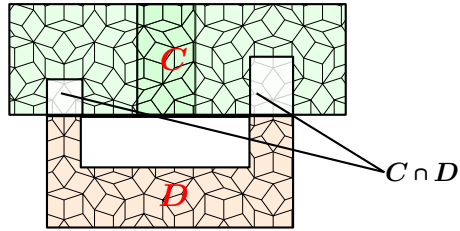


Observe that just as in the case of connected sets, unions and intersections of $\sn-$connected sets are not in general $\sn-$connected.

\begin{example}{\bf Union of $\sn-$connected sets}.\\
Consider the two disjoint balls $A$ and $B$, in Fig.~\ref{fig:disjointBalls}. Each ball by itself is $\sn-$connected, but the union of the balls is not. \qquad \textcolor{blue}{$\blacksquare$}
\end{example}

\begin{example}{\bf Intersection of $\sn-$connected sets}.\\
Take the subsets $C$ and $D$ as in Fig.~\ref{fig:strongProximalConnected}, each containing a Penrose pattern. Each ball by itself is $\sn-$connected, but it is not the same for their intersection. \qquad \textcolor{blue}{$\blacksquare$}
\end{example}

\begin{theorem}
Let $(X, \tau_X, {\sn}_X)$ and $(Y, \tau_Y, {\sn}_Y) $ be topological spaces endowed with strong proximities. Let $f: (X, \tau_X) \rightarrow (Y, \tau_Y) $ be an homeomorphism and $f: (X, {\sn}_X) \rightarrow (Y, {\sn}_Y) $ s.p.c. . Then the image of a ${\sn}_X -$connected subset is ${\sn}_Y-$connected.
\end{theorem}
\begin{proof}
Observe that $f$ preserves unions and, being an homeomorphism, $f(\Int(A_i))= \Int(f(A_i))$ for each subset. So by the hypothesis, it is easy to obtain the desired result.
\end{proof}

Recall the definition of regular open sets. A nonempty set $A$ is a \emph{regular open}, provided $A = \Int{(\cl A)}$, {\em i.e.}, $A$ is the interior of its closure. The family of regular open sets of a topological space has a nice structure. In fact, it forms a complete Boolean lattice. Furthermore, regular open sets are useful in applications because their properties seem to correspond to common-sense physical requirements (see, {\em e.g.},~\cite{Ronse1990}). Here we present some generalizations of standard results from~\cite{Willard1970}.

\begin{theorem}\label{closureconn}
Let $X$ be a topological space and $\sn$ a strong proximity on $X$. Suppose that $A = \bigcup_{i =0}^N  A_i$ is a $\sn-$connected subset and that the subsets composing the union giving $A$ are regular open subsets. Then $\cl(A)$ is $\sn-$connected.
\end{theorem}
\begin{proof}
First of all, we can write $\cl(A)= \bigcup_{i =0}^N  \cl(A_i)$, where $cl(A_i)$ is connected for each $i$. Observe that $\Int(\cl(A_i))= A_i$ by regularity and it is connected by the hypothesis. Furthermore, for each $i \geq 2$, $\cl(A_{i-1}) \sn \cl(A_i)$, because we know that $A_{i-1} \sn A_i$ and this implies $A_{i-1} \cap A_i \neq \emptyset$ by axiom $(N2)$. But again by regularity we have that this corresponds to $\Int(\cl(A_{i-1})) \cap \Int(\cl(A_i)) \neq \emptyset$. Finally, by axiom $(N4)$, we obtain $\cl(A_{i-1}) \sn \cl(A_i)$. 
\end{proof}

\begin{theorem}
Let $X$ be a topological space and $\sn$ a strong proximity on $X$. Suppose that $A = \bigcup_{i =0}^N  A_i$ is a $\sn-$connected subset and that the subsets composing the union giving $A$ are regular open subsets. If $A \subseteq G \subseteq \cl(A)$, then $G$ is $\sn-$connected.
\end{theorem}
\begin{proof}
By Theorem~\ref{closureconn}, we know that $\cl(A)$ is $\sn-$connected. It suffices to show that $G$ corresponds to $\cl_G(A)$ and then apply again the same theorem with $G$ instead of $X$. One inclusion is obvious. We need to prove that $G \subseteq \cl_G(A)$. Take $g \in G$ and $U_X(g)$ any nhbd of $g$ in $X$. Knowing that $G \subseteq \cl(A)$, we have that $U_X(g) \cap A \neq \emptyset$. Then consider $U_G(g)$, any nhbd of $g$ in $G$. It corresponds to $U_G(g)= U_X(g) \cap G$ for some $U_X(g)$. Hence we have that $U_G(g) \cap A = U_X(g) \cap G \cap A $ and being $A \subseteq G$ is equal to $U_X(g) \cap A$ that is non-empty.
\end{proof}

The next theorem provides us a tool to show that a countable topological space is $\sn-$connected.

\begin{theorem}
Suppose that $|X|\leq \aleph_0$ and  for each pair of points $x_i,\  x_{i+1} \ \in X$ there exist some nhbds $U(x_i)$ and $U(x_{i+1})$ such that they lie in some connected set of $X$ with its interior connected. Then $X$ is $\sn-$connected.
\end{theorem}
\begin{proof}
We know that for each pair of points $x_i,\  x_{i+1} \ \in X$ there exist some nhbds $U(x_i)$ and $U(x_{i+1})$ such that they lie in some connected set $A_i$ of $X$ with its interior connected. So we can write $X= \bigcup_{i \in I} A_i$. We need only to prove that $A_{i-1} \sn A_{i}$. But it follows from the fact that each $U(x_i)$ lies in $A_{i-1}$ and $A_i$. So the intersection of their interiors is non-empty and, by axiom $(N2)$, we obtain $A_{i-1} \sn A_{i}$.
\end{proof}

Now we define a kind of \emph{chain-connectedness} related to strong proximities.

\begin{definition}
Let $X$ be a topological space and $\sn$ a strong proximity on $X$. A \emph{strong chain} connecting two points $a$ and $b$ of $X$ is a sequence of open sets $U_1,...,U_n$ such that $a \in U_1$ only, $b \in U_n$ only and $U_i \sn U_j$ for $|i-j| \leq 1$.  \qquad \textcolor{blue}{$\blacksquare$}
\end{definition}

\begin{theorem}
Let $X$ be a topological space and $\sn$ a strong proximity on $X$. If $X$ is $\sn-$connected and $\mathscr{U}$ is any open cover of $X$, then each pair of points of $X$ can be connected by a strong chain consisting of elements of $\mathscr{U}$.
\end{theorem}
\begin{proof}
Immediate from Theorem~\ref{relation} and Theorem~26.15 of \cite{Willard1970}. For the sake of clarity, we give a detailed proof.

Take any $a \in X$ and consider $Z$ as the set of all points connected to $a$ by a strong chain consisting of elements of $\mathscr{U}$. Obviously $Z$ is non-empty. We want to prove that $Z$ is open and closed in $X$. So, being $X$ connected by Theorem~\ref{relation}, $Z$ coincides with the whole $X$. Let $z$ be an element of $\cl(Z)$. There exists some $U \in \mathscr{U}$ which contains $z$ and, being $U$ open, $U \cap Z $ is non-empty. So we can take an element $b$ in this intersection. This element is connected to $a$ by a strong chain in $\mathscr{U}$. If $z$ does not belong to any subset of the chain, we can join the chain with $U$. $U$ is strongly near to the last element of the chain because the elements of $\mathscr{U}$ are all open and we use axiom $(N2)$.
\end{proof}


\end{document}